\documentclass[12pt]{amsart}
\usepackage{amssymb,amsmath,amsfonts}
\usepackage{enumerate}
\usepackage[letterpaper, margin=1.2in]{geometry}

\allowdisplaybreaks

\newcommand{\N}{\mathbb N}
\newcommand{\R}{\mathbb R}
\newcommand{\Z}{\mathbb Z}

\newcommand{\rar}{\rightarrow}

\newtheorem{theorem}{Theorem}[section]

\theoremstyle{definition}

\theoremstyle{remark}

\numberwithin{equation}{section}

\begin{document}

\title[A randomized version of the Littlewood Conjecture]{A randomized version of the\\Littlewood Conjecture}
\author{Alan Haynes, Henna Koivusalo}

\thanks{Research supported by EPSRC grants EP/L001462, EP/J00149X, EP/M023540.\\ \phantom{A..}HK also gratefully acknowledges the support of Osk. Huttunen foundation. \\ \phantom{A..}MSC 2010: 11J13, 60D05}

\begin{abstract}
The Littlewood Conjecture in Diophantine approximation can be thought of as a problem about covering $\R^2$ by a union of hyperbolas centered at rational points. In this paper we consider the problem of translating the center of each hyperbola by a random amount which depends on the denominator of the corresponding rational. Using a randomized covering argument we prove that, not only is this randomized version of the Littlewood Conjecture true for almost all choices of centers, an even stronger statement with an extra factor of a logarithm also holds.
\end{abstract}
\maketitle

\section{Introduction}
The Littlewood Conjecture in Diophantine approximation asserts that, for every $\alpha,\beta\in\R$,
\[\liminf_{n\rar\infty}n\|n\alpha\|\|n\beta\|=0,\]
where $\|\cdot\|$ denotes distance to the nearest integer. It is easy to prove that the corresponding one variable statement, without the $\|n\beta\|$ factor, is false. In particular, while it is true that for all $\alpha$,
\[\liminf_{n\rar\infty}n\|n\alpha\|\le 1/2,\]
a well known result of Jarnik \cite{Jarn1928} states that the set of $\alpha\in\R$ for which
\[\liminf_{n\rar\infty}n\|n\alpha\|>0\]
has Hausdorff dimension equal to $1$. On the other hand, Khintchine's Theorem \cite{Khin1926} tells us that for almost every $\alpha\in\R$ we have that
\[\liminf_{n\rar\infty}n(\log n)\|n\alpha\|=0.\]
This shows that there is a `logarithmic gap' between what is true almost everywhere and what is true for all real numbers.\vspace*{.1in}

There is currently some discussion in the Diophantine approximation community about whether or not, in relation to the Littlewood Conjecture, such a logarithmic gap may also exist. This is precisely the content of Conjecture [L1] in \cite{BadzVela2011}, where arguments are presented in favor of this possibility. The metric analogue of Khintchine's Theorem in this setting is a result due to Gallagher \cite{Gall1962}, which implies that for almost all $(\alpha,\beta)\in\R^2$,
\begin{equation}\label{eqn.GallResult}
\liminf_{n\rar\infty}n(\log n)^2\|n\alpha\|\|n\beta\|=0.
\end{equation}
In fact a stronger statement has recently been established in \cite{BereHaynVela2016}, which tells us that for all $\alpha\in\R$, and for almost all $\beta\in\R$ (with the set of $\beta$ depending on $\alpha$), equation \eqref{eqn.GallResult} holds. The question then, is whether or not it could actually be true that for every $\alpha,\beta\in\R$,
\begin{equation}\label{eqn.StrongLittConj}
\liminf_{n\rar\infty} n(\log n)\|n\alpha\|\|n\beta\|<\infty.
\end{equation}
Although skepticism has occasionally been raised even of the original Littlewood Conjecture, there are no $\alpha$ and $\beta$ which are known to not satisfy \eqref{eqn.StrongLittConj}. The stronger statement is in fact consistent with what is known to be true for cubic irrationals from the same cubic number field \cite{CassSwin1955,Peck1961}. It is also consistent with known results for quadratic irrationals (and real numbers with quasi-periodic continued fraction expansions) in the $p$-adic versions of the Littlewood Conjecture \cite{BugeDrmodeMa2007,HaynMund2015,deMaTeul2004} (see \cite{BugeHaynVela2011} for the corresponding metric statements in this case).\vspace*{.1in}

The purpose of this note is to explain how a covering argument used by Dvoretzky in \cite{Dvor1956} can be used to prove the following randomized version of assertion \eqref{eqn.StrongLittConj}.
\begin{theorem}\label{thm.RandLitt}
  Let $(\gamma_n)_{n\in\N}$ and $(\delta_n)_{n\in\N}$ be sequences of independent random variables taking values which are uniformly distributed (with respect to Lebesgue measure) in $[0,1)$. Then almost surely we have that, for all $\alpha,\beta\in\R$,
  \[\liminf_{n\rar\infty}n(\log n)\|n\alpha-\gamma_n\|\|n\beta-\delta_n\|\le 1.\]
\end{theorem}
Of course it follows immediately from this theorem that for almost all values of $(\gamma_n)_{n\in\N}$ and $(\delta_n)_{n\in\N}$,
\[\liminf_{n\rar\infty}n\|n\alpha-\gamma_n\|\|n\beta-\delta_n\|=0,\]
for all $\alpha,\beta,\in\R$. As mentioned in the abstract, there is a natural geometric interpretation of this result. For each $n$, the set of $\alpha,\beta\in\R$ which satisfy
\[n\|n\alpha\|\|n\beta\|\le\epsilon\]
is a union of hyperbolas centered at rational points with denominators equal to $n$. In order for the Littlewood Conjecture to be true, it must be the case that, for all $\epsilon>0$, the union of all such hyperbolas covers all of $\R^2$. Our randomized result in Theorem \ref{thm.RandLitt} allows a translate, depending on $n$, of the centers of each of the hyperbolas.\vspace*{.1in}

\section{Proof of Theorem \ref{thm.RandLitt}}
Theorem \ref{thm.RandLitt} will follow from the following more general result.
\begin{theorem}\label{thm.RandLittGenForm}
 Let $(\gamma_n)_{n\in\N}$ and $(\delta_n)_{n\in\N}$ be sequences of independent random variables taking values which are uniformly distributed in $[0,1)$. Let $\psi:\N\rar[0,\infty)$ be a decreasing function and suppose that, for some $\epsilon>0$,
\begin{equation}\label{eqn.divcond}
\limsup_{n\to \infty}~\frac {1}{n^{4+\epsilon}}\cdot\exp\left((4-\epsilon)\sum_{m=1}^{n}\psi(m)\log \frac{1}{\psi(m)}\right)=\infty.
\end{equation}
Then for almost all values of $(\gamma_n)_{n\in\N}$ and $(\delta_n)_{n\in\N}$, we have that, for all $\alpha,\beta\in\R$ there are infinitely many solutions to the inequality
\begin{equation}\label{eqn.DiophIneq}
\left|\alpha+\frac{\gamma_n}{n}-\frac{a}{n}\right|\cdot\left|\beta+\frac{\delta_n}{n}-\frac{b}{n}\right|\le \frac{\psi(n)}{n^2},
\end{equation}
with $n\in\N$ and $a,b\in\Z$.
\end{theorem}
\begin{proof}
For each $n\in\N$ let
\[u_n=n^{-4-\epsilon}\cdot\exp\left((4-\epsilon)\sum_{m=1}^n\psi(m)\log\psi(m)^{-1}\right),\]
and define
\[
\Lambda=\left\{n\in\N : u_n\ge \max_{m\le n}u_m\right\}.
\]
The set $\Lambda$ is infinite by hypothesis. Furthermore if $n\in \Lambda$ then, since $u_n\ge u_{n-1}$, we obtain
\begin{equation*}
\psi(n)\log\frac{1}{\psi(n)}~\ge~ \frac{4+\epsilon}{4-\epsilon}\cdot\log\frac{n}{n-1}~=~\frac{4+\epsilon}{4-\epsilon}\cdot\frac{1}{n}+O\left(\frac{1}{n^2}\right).
\end{equation*}
Since the function $x\mapsto x\log x^{-1}$ is increasing on the interval $(0,e^{-1})$, it follows from the above equation that
\begin{equation}\label{eqfast}
\psi(n)\ge \frac{1+\epsilon/4}{n\log n},
\end{equation}
for all sufficiently large $n\in\Lambda$.\vspace*{.1in}

The problem we are considering is periodic modulo $1$ in both $\alpha$ and $\beta$. Therefore if we set
\[A_n=\{(\alpha,\beta)\in [0,1)^2:\text{\eqref{eqn.DiophIneq} holds for some } a,b\in \Z\},\]
what we are trying to prove is that
\[\limsup_{n\rar\infty}A_n=[0,1)^2.\]
If it happens that $A_n=[0,1)^2$ for infinitely many $n$ (equivalently, that $\psi(n)\ge 1/4$ for infinitely many $n$) then there is nothing to show. It follows that we can, without loss of generality, ignore all values of $n$ for which this happens. For the remaining $n$, the Lebesgue measure of $A_n$ is given by
\begin{align}
\lambda (A_n)&= n^2\left(\frac{4\psi(n)}{n^2} + 8\int _{\frac{\sqrt {\psi(n)}}{n}}^ {1/2n}\frac{\psi(n)}{n^2 \alpha}\,d\alpha\right)\nonumber\\
&=4\psi(n)\log\psi(n)^{-1}-4(\log 4-1)\psi(n).\label{eqn.measure}
\end{align}
Again, a basic computation shows that the function $x\mapsto 4x(\log x^{-1}-\log 4+1)$ is increasing on the interval $(0,1/4)$. Therefore it follows from \eqref{eqfast} that there exists an $n_0\in\N$ such that
\[\lambda (A_n)\ge\frac{4}{n},\]
for all $n\in\Lambda$ satisfying $n\ge n_0$. We suppose without loss of generality that $n_0$ is also chosen so that \eqref{eqfast} holds for all $n\in\Lambda$ with $n\ge n_0$.\vspace*{.1in}

Now suppose that $n\in\Lambda$, with $n\ge n_0$, and for each $m\le n$ define $B_m^{(n)}\subseteq A_m$ by
\[
B_m^{(n)}=\bigcup_{a,b=1}^m \left\{(\alpha,\beta)\in [0,1)^2: \left|\alpha+\frac{\gamma_m}{m}-\frac{a}{m}\right|\cdot\left|\beta+\frac{\delta_m}{m}-\frac{b}{m}\right|\le \frac{\psi(m)-1/(n\log^2 n)}{m^2}\right\}.
\]
From \eqref{eqn.measure} we have that
\begin{equation*}
\lambda(B_m^{(n)})=\lambda(A_m)+O\left(\frac{1}{n\log n}\right).
\end{equation*}
Furthermore, if there is a point $x\in[0,1)^2$ which does not lie in $\bigcup_{m=1}^nA_m$, then we have that
\[B\left(x,1/(n^2\log^2n)\right)\cap\bigcup_{m=1}^nB_m^{(n)}=\emptyset.\]
Although this is not immediately obvious, it follows easily from computing the global minimum of the distance from a point on the boundary of $A_m$, to the boundary of $B_m^{(n)}$. The global minimum occurs along the `tails' of the hyperbolic regions which form the boundaries of these sets.\vspace*{.1in}

What we have shown so far implies that, if there is a point $x\notin \bigcup_{m=1}^ n A_m$, then, as long as $n$ is large enough (depending on $\epsilon$), there exist integers $1\le a,b< \lfloor n^{2+\epsilon/2}\rfloor$ satisfying
\[\left(\frac{a}{\lfloor n^{2+\epsilon/2}\rfloor},\frac{b}{\lfloor n^{2+\epsilon/2}\rfloor}\right)\notin\bigcup_{m=1}^n B_m^{(n)}.\]
Using our independence assumption, the measure of the set of $(\gamma_m,\delta_m)$ pairs for which such integers $a$ and $b$ exist is bounded above by
\begin{align}
&n^{4+\epsilon}\cdot\lambda\left(\bigcap_{m=1}^n([0,1)^2\setminus B_m^{(n)})\right)\nonumber\\
&\qquad =n^{4+\epsilon}\prod _{m=1}^n(1-\lambda(B_m^{(n)}))\nonumber\\
&\qquad = n^{4+\epsilon}\prod_{m=1}^n \left(1-\lambda(A_m)+O\left(\frac {1}{n\log n}\right)\right).\label{eqn.measurebound}
\end{align}
We assume at this point that $\lambda(A_m)\rar 0$ as $m\rar\infty$ (if this is not the case, the remainder of the argument is very easy). Under this assumption we have, for $n\in\Lambda$ sufficiently large, that the expression on the right hand side of \eqref{eqn.measurebound} is bounded above by a constant times
\begin{align*}
n^{4+\epsilon}\cdot\exp\left(-\sum_{m=1}^n\lambda(A_m)\right).
\end{align*}
This in turn is bounded above, again for $n$ sufficiently large, by a constant times
\begin{align*}
n^{4+\epsilon}\cdot\exp\left(-(4-\epsilon)\sum_{m=1}^n\psi(m)\log\psi(m)^{-1}\right),
\end{align*}
and, by hypothesis, this expression tends to $0$ as $n\rar\infty$. This shows that, for almost every choice of $(\gamma_n)_{n\in\N}$ and $(\delta_n)_{n\in\N}$, we have that
\[\bigcup_{m=1}^\infty A_m=[0,1)^2.\]
The same argument also shows that, for any $M\in\N$, we have almost surely that
\[\bigcup_{m=M}^\infty A_m=[0,1)^2,\]
and this completes the proof.\vspace*{.1in}
\end{proof}
To verify the statement of Theorem \ref{thm.RandLitt}, let $\delta>0$ and take
\[\psi(n)=\frac{1+\delta}{n\log(n+1)}.\]
It is not difficult to check that we can then choose $\epsilon>0$ so that \eqref{eqn.divcond} is satisfied, and we conclude from Theorem \ref{thm.RandLittGenForm} that, for almost every choice of $(\gamma_n)_{n\in\N}$ and $(\delta_n)_{n\in\N}$,
\[\liminf_{n\rar\infty}n(\log n)\|n\alpha-\gamma_n\|\|n\beta-\delta_n\|\le 1+\delta.\]
Since $\delta>0$ is arbitrary, the claim of Theorem \ref{thm.RandLitt} follows.

\vspace{.15in}

{\footnotesize
\noindent
AH: Department of Mathematics, University of Houston,\\
Houston, TX, United States.\\
haynes@math.uh.edu\\

\noindent
HK: Faculty of Mathematics, University of Vienna,\\
Oskar Morgensternplatz 1, 1090 Vienna, Austria.\\
henna.koivusalo@univie.ac.at
}

\end{document}